\documentclass[preprint,review,10pt]{amsart}
\usepackage{amsmath,amssymb,amsfonts,xspace}
\newtheorem{theorem}{Theorem}[section]

\theoremstyle{definition}
\newtheorem{definition}[theorem]{Definition}
\newtheorem{example}[theorem]{Example}

\newtheorem{corollary}[theorem]{Corollary}
\newtheorem{lem}[theorem]{Lemma}

\theoremstyle{remark}

\numberwithin{equation}{section}

\begin{document}

\title{extensions of $n$-ary prime hyperideals via an $n$-ary multiplicative subset in a Krasner $(m,n)$-hyperring
  }

%    Information for first author
\author{M. Anbarloei}
%    Address of record for the research reported here
\address{Department of Mathematics, Faculty of Sciences,
Imam Khomeini International University, Qazvin, Iran.
}
%    Current address

\email{m.anbarloei@sci.ikiu.ac.ir }
%    \thanks will become a 1st page footnote.

%    Information for second author
%\author{}
%\address{}
%\email{}
%\thanks{Support information for the second author.}

%    General info
\subjclass[2010]{ 16Y99}

%\date{September  , 2013.}

%\dedicatory{This paper is dedicated to our advisors.}
\keywords{  $n$-ary $J$-hyperideal, $n$-ary $\delta$-$J$-hyperideal, $(k,n)$-absorbing $\delta$-$J$-hyperideal.}
%------------------------------------------------------------------------------
%%%%%%%%%%%%%%%%%%%%%%%%%%%%%%%%%%%%%%%%%%%%%%%%%%%%%%%%%%%%%%%%%%%%%%%%%%%%%%%%%%%%%%%%%%%%%%%%%%%%%%

%%%%%%%%%%%%%%%%%%%%%%%%%%%%%%%%%%%%%%%%%%%%%%%%%%%%%%%%%%%%%%%%%%%%%%%%%%%%%%%%%%%%%%%%%%%%%%%%%%%%%%%%%%%%%%%%%%%%%%%%%%%%%%%%%%%%%%%%%
\begin{abstract}
Let $R$ be a Krasner $(m,n)$-hyperring and $S$ be an n-ary multiplicative subset of $R$. The purpose of this paper is to introduce the notion of n-ary $S$-prime hyperideals as a new expansion of n-ary prime hyperideals. A hyperideal $I$ of $R$ disjoint with $S$ is said to be an n-ary $S$-prime hyperideal if there exists  $s \in S$ such that whenever $g(x_1^n) \in I$ for all $x_1^n \in R$, then  $g(s,x_i,1^{(n-2)}) \in I$ for some $1 \leq i \leq n$. Several properties and characterizations concerning n-ary $S$-prime hyperideals are presented. The stability of this new concept with respect to various hyperring-theoretic constructions are studied. Furthermore, we extend this concept to n-ary $S$-primary hyperideals. We obtained some specific results explaining the structure. 

\end{abstract}
%%%%%%%%%%%%%%%%%%%%%%%%%%%%%%%%%%%%%%%%%%%%%%%%%%%%%%%%%%%%%%%%%%%%%%%%%%%%%%%%%
\maketitle
\section{Introduction}
Prime and primary ideals which are quite important in commutative rings have been studied by many authors. In 2019,  Hamed and Malek \cite{hamed} introduced the notion of $S$-prime ideal which is a generalization of prime ideals. Suppose that $R$ is a commutative ring with identity and  $S \subseteq R$ a multiplicative subset. A proper ideal $P$ of $R$ disjoint from $S$ is called  an $S$-prime of $R$ if there exists an $s \in S$ such that for all $x,y \in R$  if $xy \in P$, then $sx \in P$ or $sy \in P$. In \cite{massaoud} , Massaoud defined and investigated the  concept of $S$-primary ideals of a commutative ring in a way that generalizes essentially all the results concerning primary ideals. A proper ideal $Q$ of $R$ disjoint from $S$ is called  an $S$-primary of $R$ if there exists an $s \in S$ such that for all $x,y \in R$  if $xy \in P$, then $sx \in P$ or $sy \in \sqrt{Q}$. Furthermore, some results on S-primary ideals of a commutative ring were studied by Visweswaran in \cite{visweswaran}.

Hyperstructures represent a natural extension of classical algebraic structures and they were
defined by the French mathematician F. Marty.   In 1934,  Marty \cite{s1} defined  the concept of a hypergroup as a generalization of groups during the $8^{th}$ Congress of the Scandinavian Mathematicians. A comprehensive
review of the theory of hyperstructures can be found in  \cite {s2, s3, davvaz1, davvaz2, s4, jian}. The simplest algebraic hyperstructures which possess the properties of closure and associativity are called  semihypergroups. 
$n$-ary semigroups and $n$-ary groups are algebras with one $n$-ary operation which is associative and invertible in a generalized sense. The notion of investigations of $n$-ary algebras goes back to Kasner’s lecture \cite{s5} at a  scientific meeting in 1904.   In 1928, Dorente wrote the first paper concerning the theory of $n$-ary groups \cite{s6}. Later on, Crombez and
Timm \cite{s7, s8} defined and described  the notion of the $(m, n)$-rings and their quotient structures. Mirvakili and Davvaz [20] defined $(m,n)$-hyperrings and obtained several results in this respect. In \cite{s9}, they introduced and illustrated a generalization of the notion of a hypergroup in the sense of Marty and a generalization of an $n$-ary group,   which is called $n$-ary hypergroup.
 The $n$-ary structures  has been studied in  \cite{l1, l2, l3, ma, rev1}. Mirvakili and Davvaz \cite{cons} defined $(m,n)$-hyperrings and obtained several results in this respect.

It was Krasner, who introduced one important class of hyperrings, where the addition is a hyperoperation, while the multiplication is an ordinary binary operation, which is called Krasner hyperring.  In \cite{d1},  a generalization of the Krasner hyperrings, which is a subclass of $(m,n)$-hyperrings, was defined by Mirvakili and Davvaz. It is called Krasner $(m,n)$-hyperring. Ameri and Norouzi in \cite{sorc1} introduced some important
hyperideals such as Jacobson radical, n-ary prime and primary hyperideals, nilradical, and n-ary multiplicative subsets of Krasner $(m, n)$-hyperrings. Afterward, the notions of $(k,n)$-absorbing hyperideals and $(k,n)$-absorbing primary hyperideals were studied by Hila et. al. \cite{rev2}. 
Norouzi et. al.  proposed and analysed a new defnition for normal hyperideals in Krasner $(m,n)$-hyperrings, with respect to that one given in \cite{d1} and they showed that these hyperideals correspond to strongly regular relations \cite{nour}.  Asadi and Ameri  introduced and studied direct limit of a direct system in the category of Krasner $(m,n)$-hyperrigs \cite{asadi}. 
Dongsheng  defined the notion of $\delta$-primary ideals in  a commutative ring  where  $\delta$ is a function that assigns to each ideal $I$  an ideal $\delta(I)$ of the same ring \cite{bmb2}.  Moreover, in \cite{bmb3} he and his colleague  investigated 2-absorbing $\delta$-primary ideals  which  unify 2-absorbing ideals and 2-absorbing primary
ideals. Ozel Ay et al.  generalized the notion of $\delta$-primary  on Krasner hyperrings \cite{bmb4}. The concept of $\delta$-primary  hyperideals in Krasner $(m,n)$-hyperrings, which unifies the prime and primary hyperideals under one frame, was defined in \cite{mah3}. \\In this paper, we aim to  complete this circle of ideas.  
Motivated by the research works on $S$-prime ideals and $S$-primary ideals of  commutative rings,  we define and investigate  the notions of n-ary $S$-prime and n-ary $S$-primary hyperideals in a commutative Krasner $(m,n)$-hyperring. 
%%%%%%%%%%%%%%%%%%%%%%
\section{Preliminaries}
In this section we recall some definitions and results concerning $n$-ary hyperstructures which we need to develop our paper.\\
Let $H$ be a nonempty set and $P^*(H)$ be the
set of all the non-empty subsets of $H$.  Then the mapping $f : H^n \longrightarrow P^*(H)$
 is called an $n$-ary hyperoperation and the algebraic system $(H, f)$ is called an $n$-ary hypergroupoid. For non-empty subsets $A_1,..., A_n$ of $H$ we define

$f(A^n_1) = f(A_1,..., A_n) = \bigcup \{f(x^n_1) \ \vert \  x_i \in  A_i, i = 1,..., n \}$.\\
The sequence $x_i, x_{i+1},..., x_j$ 
will be denoted by $x^j_i$. For $j< i$, $x^j_i$ is the empty symbol. Using this notation,

$f(x_1,..., x_i, y_{i+1},..., y_j, z_{j+1},..., z_n)$ \\
will be written as $f(x^i_1, y^j_{i+1}, z^n_{j+1})$. The  expression will be written in the form $f(x^i_1, y^{(j-i)}, z^n_{j+1})$, when $y_{i+1} =... = y_j = y$ . 
 
 If for every $1 \leq i < j \leq n$ and all $x_1, x_2,..., x_{2n-1} \in H$, 

$f(x^{i-1}_1, f(x_i^{n+i-1}), x^{2n-1}_{n+i}) = f(x^{j-1}_1, f(x_j^{n+j-1}), x_{n+j}^{2n-1}),$ \\
then the n-ary hyperoperation $f$ is called associative. An $n$-ary hypergroupoid with the
associative $n$-ary hyperoperation is called an $n$-ary semihypergroup. 

An $n$-ary hypergroupoid $(H, f)$ in which the equation $b \in f(a_1^{i-1}, x_i, a_{ i+1}^n)$ has a solution $x_i \in H$
for every $a_1^{i-1}, a_{ i+1}^n,b  \in H$ and $1 \leq i \leq n$, is called an $n$-ary quasihypergroup, when $(H, f)$ is an $n$-ary
semihypergroup, $(H, f)$ is called an $n$-ary hypergroup.  

An $n$-ary hypergroupoid $(H, f)$ is commutative if for all $ \sigma \in \mathbb{S}_n$, the group of all permutations of $\{1, 2, 3,..., n\}$, and for every $a_1^n \in H$ we have $f(a_1,..., a_n) = f(a_{\sigma(1)},..., a_{\sigma(n)})$.
  If  $a_1^n \in H$ then we denote $a_{\sigma(1)}^{\sigma(n)}$ as the $(a_{\sigma(1)},..., a_{\sigma(n)})$.

If $f$ is an $n$-ary hyperoperation and $t = l(n- 1) + 1$, then $t$-ary hyperoperation $f_{(l)}$ is given by

$f_{(l)}(x_1^{l(n-1)+1}) = f(f(..., f(f(x^n _1), x_{n+1}^{2n -1}),...), x_{(l-1)(n-1)+1}^{l(n-1)+1})$. 
\begin{definition}
(\cite{d1}). Let $(H, f)$ be an $n$-ary hypergroup and $B$ be a non-empty subset of $H$. $B$ is called
an $n$-ary subhypergroup of $(H, f)$, if $f(x^n _1) \subseteq B$ for $x^n_ 1 \in B$, and the equation $b \in f(b^{i-1}_1, x_i, b^n _{i+1})$ has a solution $x_i \in B$ for every $b^{i-1}_1, b^n _{i+1}, b \in B$ and $1 \leq i  \leq n$.
An element $e \in H$ is called a scalar neutral element if $x = f(e^{(i-1)}, x, e^{(n-i)})$, for every $1 \leq i \leq n$ and
for every $x \in H$. 

An element $0$ of an $n$-ary semihypergroup $(H, g)$ is called a zero element if for every $x^n_2 \in H$ we have
$g(0, x^n _2) = g(x_2, 0, x^n_ 3) = ... = g(x^n _2, 0) = 0$.
If $0$ and $0^ \prime $are two zero elements, then $0 = g(0^ \prime , 0^{(n-1)}) = 0 ^ \prime$  and so the zero element is unique. 
\end{definition}
\begin{definition}
\cite{l1} Let $(H, f)$ be a  $n$-ary hypergroup. $(H, f)$ is called a canonical $n$-ary
hypergroup if\\
(1) there exists a unique $e \in H$, such that for every $x \in H, f(x, e^{(n-1)}) = x$;\\
(2) for all $x \in H$ there exists a unique $x^{-1} \in H$, such that $e \in f(x, x^{-1}, e^{(n-2)})$;\\
(3) if $x \in f(x^n _1)$, then for all $i$, we have $x_i \in  f(x, x^{-1},..., x^{-1}_{ i-1}, x^{-1}_ {i+1},..., x^{-1}_ n)$.

We say that $e$ is the scalar identity of $(H, f)$ and $x^{-1}$ is the inverse of $x$. Notice that the inverse of $e$ is $e$.
\end{definition}
\begin{definition}
\cite{d1} A Krasner $(m, n)$-hyperring is an algebraic hyperstructure $(R, f, g)$, or simply $R$,  which
satisfies the following axioms:\\
(1) $(R, f$) is a canonical $m$-ary hypergroup;\\
(2) $(R, g)$ is a $n$-ary semigroup;\\
(3) the $n$-ary operation $g$ is distributive with respect to the $m$-ary hyperoperation $f$ , i.e., for every $a^{i-1}_1 , a^n_{ i+1}, x^m_ 1 \in R$, and $1 \leq i \leq n$,

$g(a^{i-1}_1, f(x^m _1 ), a^n _{i+1}) = f(g(a^{i-1}_1, x_1, a^n_{ i+1}),..., g(a^{i-1}_1, x_m, a^n_{ i+1}))$;\\
(4) $0$ is a zero element (absorbing element) of the $n$-ary operation $g$, i.e., for every $x^n_ 2 \in R$ we have 

$g(0, x^n _2) = g(x_2, 0, x^n _3) = ... = g(x^n_ 2, 0) = 0$.
\end{definition}
We assume throughout this paper that all Krasner $(m,n)$-hyperrings are commutative.

A non-empty subset $S$ of $R$ is called a subhyperring of $R$ if $(S, f, g)$ is a Krasner $(m, n)$-hyperring. Let
$I$ be a non-empty subset of $R$, we say that $I$ is a hyperideal of $(R, f, g)$ if $(I, f)$ is an $m$-ary subhypergroup
of $(R, f)$ and $g(x^{i-1}_1, I, x_{i+1}^n) \subseteq I$, for every $x^n _1 \in  R$ and  $1 \leq i \leq n$.

\begin{definition} \cite{sorc1} For every element $x$ in a Krasner $(m,n)$-hyperring $R$, the hyperideal generated by $x$ is denoted by $<x>$ and defined as follows:

$<x>=g(R,x,1^{(n-2)})=\{g(r,x,1^{(n-2)}) \ \vert \ r \in R\}$
\end{definition}
%\begin{definition} \cite{sorc1}
%A hyperideal $M$ of a Krasner $(m, n)$-hyperring $R$
%is said to be maximal if for every hyperideal $N$ of $R$, $M \subseteq N \subseteq R$ implies that $N=M$ or $N=R$.
%\end{definition}
%The Jacobson radical of a Krasner $(m, n)$-hyperring $R$
%is the intersection of all maximal hyperideals of $R$ and it is denoted by $J_{(m,n)}(R)$. If $R$ does not have any maximal hyperideal, we let $J_{(m,n)}(R)=R$.

\begin{definition} \cite{sorc1}
We say that an element $x \in  R$ is invertible  if there exists $y \in R$ such that $1_R=g(x,y,1_R^{(n-2)})$. Also,
the subset $U$ of $R$ is invertible  if and only if every element of $U$ is invertible .
\end{definition}
\begin{definition}
 \cite{sorc1} A hyperideal $P$ of a Krasner $(m, n)$-hyperring $R$, such that $P \neq R$, is called a prime hyperideal if for hyperideals $U_1,..., U_n$ of $R$, $g(U_1^ n) \subseteq P$ implies that $U_1 \subseteq P$ or $U_2 \subseteq P$ or ...or $U_n \subseteq P$.
\end{definition}
\begin{lem} 
(Lemma 4.5 in \cite{sorc1})Let $P\neq R$ be a hyperideal of a Krasner $(m, n)$-hyperring $R$. Then $P$ is a prime hyperideal if for all $x^n_ 1 \in R$, $g(x^n_ 1) \in P$ implies that $x_1 \in P$ or ... or $x_n \in P$. 
\end{lem}
\begin{definition} (\cite{sorc1}) Let $I$ be a hyperideal in a  Krasner $(m, n)$-hyperring $R$ with
scalar identity. The radical (or nilradical) of $I$, denoted by ${\sqrt I}^{(m,n)}$
is the hyperideal $\bigcap P$, where
the intersection is taken over all  prime hyperideals $P$ which contain $I$. If the set of all prime hyperideals containing $I$ is empty, then ${\sqrt I}^{(m,n)}$ is defined to be $R$.
\end{definition}
 Ameri and  Norouzi showed that if $x \in {\sqrt I}^{(m,n)}$ then 
 there exists $t \in \mathbb {N}$ such that $g(x^ {(t)} , 1_R^{(n-t)} ) \in I$ for $t \leq n$, or $g_{(l)} (x^ {(t)} ) \in I$ for $t = l(n-1) + 1$ \cite{sorc1}.
 
\begin{definition}
\cite{sorc1} A proper hyperideal $I$ in a  Krasner $(m, n)$-hyperring $R$ with the
scalar identity $1_R$ is said to be  primary if $g(x^n _1) \in I$ and $x_i \notin I$ implies that $g(x_1^{i-1}, 1_R, x_{ i+1}^n) \in {\sqrt I}^{(m,n)}$ for some $1 \leq i \leq n$.
\end{definition}
 If $I$ is a primary hyperideal in a  Krasner $(m, n)$-hyperring $R$ with the scalar identity $1_R$, then ${\sqrt I}^{(m,n)}$ is  prime. (Theorem 4.28 in \cite{sorc1})
 \begin{definition} \cite{sorc1}
 A non-empty subset $S$ of a Krasner $(m,n)$-hyperring $R$ is called an n-ary multiplicative, if  $g(s_1^n) \in S$ for $s_1,...,s_n \in S$.
 \end{definition}
\begin{definition} \cite{d1}
Let $(R_1, f_1, g_1)$ and $(R_2, f_2, g_2)$ be two Krasner $(m, n)$-hyperrings. A mapping
$h : R_1 \longrightarrow R_2$ is called a homomorphism if for all $x^m _1 \in R_1$ and $y^n_ 1 \in R_1$ we have

$h(f_1(x_1,..., x_m)) = f_2(h(x_1),...,h(x_m))$

$h(g_1(y_1,..., y_n)) = g_2(h(y_1),...,h(y_n)). $
\end{definition}

%333333333333333333333333333333333333333333333
%33333333333333333333333333333333333333333333333
%3333333333333333333333333333333333333333333333
%33333333333333333333333333333333333333333333333
\section{$n$-ary S-prime hyperideals}
We start this section by introducing the concept of n-ary $S$-prime hyperideals of  Krasner $(m,n)$-hyperring $R$ where $S$ is an n-ary multiplicative subset of $R$. The following definition constitutes the $S$-version of n-ary prime hyperideals.
\begin{definition}
Let $S$ be  an n-ary multiplicative subset of a Krasner $(m,n)$-hyperring $R$ and $I$ be a hyperideal of $R$ with $I \cap S=\varnothing$. $I$ refers to an n-ary $S$-prime hyperideal if there exists an $s \in S$ such that for all $x_1^n \in R$ with $g(x_1^n) \in I$, we get $g(s,x_i,1^{(n-2)}) \in I$ for some $1 \leq i \leq n$.
\end{definition}
\begin{example} \label{ex}
The set $R=\{0,1,2\}$ with the following 3-ary hyperoeration $f$ and 3-ary operation $g$ is a Krasner $(3,3)$-hyperring such that $f$ and $g$ are commutative.
\[f(0,0,0)=0, \ \ \ f(0,0,1)=1, \ \ \ f(0,1,1)=1, \ \ \ f(1,1,1)=1, \ \ \ f(1,1,2)=R,\]
\[f(0,1,2)=R, \ \ \ f(0,0,2)=2,\ \ \ f(0,2,2)=2,\ \ \ f(1,2,2)=R, \ \ \ f(2,2,2)=2,\]
$\ \ \ g(1,1,1)=1,\ \ \ \ g(1,1,2)=g(1,2,2)=g(2,2,2)=2,$\\
and for $x_1,x_2 \in R, \ g(0,x_1,x_2)=0$.

Consider   3-ary multiplicative subset $S=\{1,2\}$ of Krasner $(3, 3)$-hyperring $(R, f , g)$. Then  hyperideal $P=\{0,2\}$ is a 3-ary $S$-prime hyperideal of $R$.
\end{example}
The following example shows that an n-ary $S$-prime hyperideal may not be an n-ary prime hyperideal of $R$. 
\begin{example} 
The set $R=\{0,1,2,3\}$ with following 2-hyperoperation $"\oplus"$ is a canonical 2-ary hypergroup.

\hspace{1.5cm}
\begin{tabular}{c|c} 
$\oplus$ & $0$ \ \ \ \ \ \ \ $1$ \ \ \ \ \ \ \ $2$ \ \ \ \ \ \ \ $3$
\\ \hline 0 & $0$\ \ \ \ \ \ \ $1$\ \ \ \ \ \ \ \ \ $2$ \ \ \ \ \ \ \ $3$ 
\\ $1$ & $1$ \ \ \ \ \ \ \ $A$ \ \ \ \ \ \ \ $3$ \ \ \ \ \ \ $B$
\\ $2$ & $2$ \ \ \ \ \ \ \ $3$ \ \ \ \ \ \ \ $0$ \ \ \ \ \ \ \ $1$
\\ $3$ & $3$ \ \ \ \ \ \ \ $B$ \ \ \ \ \ \ \ $1$ \ \ \ \ \ \ \ $A$
\end{tabular}

In which $A=\{0,1\}$ and $B=\{2,3\}$. Define a 4-ary operation $g$ on $R$ as follows:
If $a_1,a_2,a_3,a_4 \in B$ then $g(a_1^4)=2$ and otherwise $g(a_1^4)=0$.
 It follows that $(R,\oplus,g)$ is a Krasner (2,4)-hyperring. $S=\{2,3\}$ is a 4-ary multiplicative subset of $R$. In the hyperring, $I=\{0\}$ is a 4-ary $S$-prime hyperideal of $R$ but it is not  prime, because $g(1,2,2,3)=0 \in I$ while $1,2,3 \notin I$.
\end{example}
Our first theorem gives a characterization of n-ary $S$-prime hyperideals.
\begin{theorem} \label{31}
Let $S$ be  an n-ary multiplicative subset of a Krasner $(m,n)$-hyperring $R$ and $I$ be a hyperideal of $R$ with $I \cap S=\varnothing$. Then $I$ is  n-ary $S$-prime if and only if $(I:s)=\{r \in R \ \vert \ g(r,s,1^{(n-2)}) \in I\}$ is an n-ary prime hyperideal of $R$ for some $s \in S$.
\end{theorem}
\begin{proof}
$\Longrightarrow$ Let $I$ be is an n-ary $S$-prime hyperideal of $R$. Then there exists $s \in S$ such that for all $x_1^n \in R$ with $g(x_1^n) \in I$, we get $g(s,x_i,1^{(n-2)}) \in I$ for some $1 \leq i \leq n$. Suppose that  $g(y_1^n) \in (I:s)$ for $y_1^n \in R$. Then $g(s,g(y_1^n),1^{(n-2)})=g(g(s,y_1,1^{(n-2)}),y_2^n) \in I$ which means $g(s,g(s,y_1,1^{(n-2)}),1^{(n-2)})=g(s^2,y_1,1^{(n-3)}) \in I$ or $g(s,y_i,1^{(n-2)}) \in I$ for some $2 \leq i \leq n$. Since $I \cap S=\varnothing$, then we conclude that $g(s,y_1,1^{(n-2)}) \in I$ or $g(s,y_i,1^{(n-2)}) \in I$ for some $2 \leq i \leq n$. So, $g(s,y_i,1^{(n-2)}) \in I$ for some $1 \leq i \leq n$ which implies $y_i \in (I:s)$ for some $1 \leq i \leq n$. Consequently, $(I:s)$ is an n-ary prime hyperideal of $R$.\\
$\Longleftarrow$ Let $(I:s)$ be an n-ary prime hyperideal of $R$ for some $s \in S$. Suppose that $g(x_1^n)\in I$ for $x_1^n \in R$. Since $I \subseteq (I:s)$, then $g(x_1^n) \in (I:s)$. Since $(I:s)$ is an n-ary prime hyperideal of $R$, then $x_i \in (I:s)$ for some $1 \leq i \leq n$. This implies that $g(s,x_i,1^{(n-2)})\in I)$ for some $1 \leq i \leq n$ which means $I$ is  an n-ary $S$-prime hyperideal of $R$.
\end{proof}
\begin{theorem}
Let $S$ be  an n-ary multiplicative subset of a Krasner $(m,n)$-hyperring $R$ and  $R \subseteq G$ be  an extension of $R$. If $I$ is an n-ary $S$-prime hyperideal of $G$, then $I \cap R$ is an n-ary $S$-prime hyperideal of $R$.
\end{theorem}
\begin{proof}
 Let $I$ be  an n-ary $S$-prime hyperideal of $G$. Then there exist $s \in S$ such that for all $x_1^n \in R$ with $g(x_1^n) \in I$, we get $g(s,x_i,1^{(n-2)}) \in I$ for some $1 \leq i \leq n$. Let $g(x_1^n) \in I \cap R$ for $x_1^n \in R$. Since $g(x_1^n) \in I$, then $g(s,x_i,1^{(n-2)}) \in I$ for some $1 \leq i \leq n$ which means $g(s,x_i,1^{(n-2)}) \in I \cap R$. Thus,  $I \cap R$ is an n-ary $S$-prime hyperideal of $R$.
\end{proof}
\begin{theorem}  \label{32}
Let $S$ be  an n-ary multiplicative subset of a Krasner $(m,n)$-hyperring $R$ and $I$ be a hyperideal of $R$. If $I \subseteq \cup_{i=1}^n I_i$ for some n-ary $S$-prime hyperideals $I_1^n$ of $R$, then there exists $s \in S$ such that $g(s,I,1^{(n-2)}) \subseteq P_i$ for some $1 \leq i \leq n$.
\end{theorem}
\begin{proof}
Let $I \subseteq \cup_{i=1}^n I_i$ for some n-ary $S$-prime hyperideals $I_1^n$ of $R$. For each $1 \leq i \leq n$, we get $s_i \in S$ such that $(I_i : s_i)$ is an n-ary prime hyperideal of $R$, by Theorem \ref{31}. Since $I \subseteq \cup_{i=1}^n I_i\subseteq \cup_{i=1}^n(I_i: s_i)$, we have $I \subseteq (I_i : s_i)$ for some $1 \leq i \leq n$, by Theorem 5.1 in \cite{mah4}. Thus $g(s_i,I,1^{(n-2)}) \subseteq I_i$.
\end{proof}
\begin{theorem}  \label{33}
Let $S$ be  an n-ary multiplicative subset of a Krasner $(m,n)$-hyperring $R$ and $I$ be a hyperideal of $R$ with $I \cap S=\varnothing$. Then $I$ is  n-ary $S$-prime if and only if there exists $s \in S$, for all hyperideals $I_1^n$ of $R$, if $g(I_1^n) \subseteq I$, then $g(s,I_i,1^{(n-1)}) \subseteq I$ for some $1 \leq i \leq n$.
\end{theorem}
\begin{proof}
$\Longrightarrow$ Let $I$  is  an n-ary $S$-prime hyperideal of $R$. Then there exists $s \in S$ such that for all $x_1^n \in R$ with $g(x_1^n) \in I$ we have $g(s,x_i,1^{(n-2)}) \in I$ for some $1 \leq i \leq n$. Let for some hyperideals $I_1^n$ of $R$ with $g(I_1^n) \in I$ we have $g(s,I_i,1^{(n-2)}) \nsubseteq I$ for all $1 \leq i \leq n$. This means $g(s,a_i,1^{(n-2)}) \notin I$ for some $a_i \in I_i$ and $1 \leq i \leq n$ which is a contradiction, since $I$ is an n-ary $S$-prime hyperideal of $R$ and $g(a_1^n) \in g(I_1^n) \subseteq I$. \\
$\Longleftarrow$   Let $g(x_1^n) \in I$ for $x_1^n \in R$. Then $ g(\langle x_1 \rangle, \cdots, \langle x_n \rangle) \subseteq I$. Hence we have $g(s,\langle x_i \rangle, 1^{(n-2)}) \subseteq I$ for some $1 \leq i \leq n$ which implies $g(s,x_i,1^{(n-2)}) \in I$. Thus, $I$ is  an n-ary $S$-prime hyperideal of $R$.
\end{proof} 
In view of Theorem \ref{33}, we have the following result.
\begin{corollary}  \label{34}
Let $I$ be a proper hyperideal of a Krasner $(m,n)$-hyperring $R$. Then $I$ is an n-ary prime hyperideal if and only if for all hyperideals $I_1^n$ of $R$, If $g(I_1^n) \subseteq I$, then $I_i \subseteq I$ for some $1 \leq i \leq n$.
\end{corollary}
\begin{proof}
Consider $S=\{1\}$. Then we are done, by Theorem \ref{33}.
\end{proof}
\begin{theorem} \label{35}
Let $S$ be  an n-ary multiplicative subset of a Krasner $(m,n)$-hyperring $R$ and $I$ be an n-ary $S$-prime hyperideal of $R$. If $J$ be a hyperideal of $R$ with $J \subseteq I$, then $g(s, \sqrt{J}^{(m,n)},1^{(n-2)}) \subseteq I$ for some $s \in S$.
\end{theorem}
\begin{proof}
Let $a \in \sqrt{J}^{(m,n)}$. Then there exists $t \in \mathbb{N}$ such that $g(a^{(t)},1^{(n-t)}) \in J$ for $t \leq n$ or $g_{(l)}(a^{(t)}) \in I$ for $t=l(n-1)+1$. If $g(a^{(t)},1^{(n-t)}) \in J \subseteq I$, then $g(\langle a \rangle ^{(t)},1^{(n-t)}) \subseteq I$. By Theorem \ref{33}, we get $g(s,\langle a \rangle,1^{(n-2)}) \subseteq I$ for some $s \in S$ which implies $g(s,a,1^{(n-2)}) \in I$. Consequently, $g(s,\sqrt{J}^{(m,n)}, 1^{(n-2)}) \subseteq I$.  If $t = l(n-1)+1$, then by using a similar argument, one can easily complete the proof.
\end{proof}
\begin{theorem}  \label{36}
Let $S$ be  an n-ary multiplicative subset of a Krasner $(m,n)$-hyperring $R$ and $I_1^n$ be  n-ary $S$-prime hyperideals of $R$. Then there exists $s \in S$ such that $g(s,\sqrt{\cap_{i=1}^n I_i},1^{(n-2)}) \subseteq \cap_{i=1}^n I_i$.
\end{theorem}
\begin{proof}
Let $I_1^n$ be  n-ary $S$-prime hyperideals of $R$. Then for each $1 \leq i \leq n$ we have $s_i \in S$ such that for all $x_1^n$ of $R$, if $g(x_1^n) \in I_i$, then $g(s_i,x_j,1^{(n-2)}) \in I_i$ for some $1 \leq j \leq n$. By Theorem \ref{35}, we get $g(s_i,\sqrt{I_i}^{(m,n)},1^{(n-2)}) \subseteq I_i$  for each $1 \leq i \leq n$. Put $s=g(s_1^n)$. Hence we obtain $g(s,\sqrt{\cap_{i=1}^n I_i},1^{(n-2)}) =g(s,\cap_{i=1}^n \sqrt{I_i}^{(m,n)},1^{(n-2)}) \subseteq \cap_{i=1}^n I_i$.
\end{proof}
\begin{theorem}  \label{37}
Let $(R_1,f_1,g_1),  (R_2,f_2,g_2)$ be  Krasner $(m,n)$-hyperrings and $h:R_1 \longrightarrow R_2$ be a homomorphism such that $0 \notin h(S)$. If $I_2$ is an n-ary $h(S)$-prime hyperideal of $R_2$, then $h^{-1}(I_2)$ is an n-ary $S$-prime hyperideal of $R_1$.
\end{theorem}
\begin{proof}
Suppose that $I_2$ is an n-ary $h(S)$-prime hyperideal of $R_2$. Then there exists $s \in S$ such that for all $y_1^n \in R_2$ with $g_2(y_1^n) \in I_2$ we have $g_2(h(s),y_i,1^{(n-2)}) \in I_2$ for some $1 \leq i \leq n$. Put $I_1=h^{-1}(I_2)$. It is easy to see that $I_1 \cap S = \varnothing$. Let $g_1(x_1^n) \in I_1$ for $x_1^n \in R_1$. Then $h(g_1(x_1^n))=g_2(h(x_1),...,h(x_n)) \in I_2$. So, we have $g_2(h(s),h(x_i),1^{(n-2)})=h(g_1(s,x_i,1^{(n-2)}) \in I_2$ for some $1 \leq i \leq n$ which implies $g_1(s,x_i,1^{(n-2)}) \in h^{-1}(I_2)=I_1$. Consequently, $h^{-1}(I_2)$ is an n-ary $S$-prime hyperideal of $R_1$.
\end{proof}
The concept of Krasner $(m,n)$-hyperring of fractions  was introduced in \cite{mah5}.
\begin{theorem} 
Let $R$ be a Krasner $(m,n)$-hyperring and $S$ be an $n$-ary multiplicative subset of $R$ with $1 \
\in S$. If $I$ is an $n$-ary $S$-prime hyperideal of $R$ with $I \cap S=\varnothing$, then $S^{-1}I$ is an $n$-ary prime hyperideal of $S^{-1}R$.
\end{theorem}
\begin{proof}
Let $G(\frac{a_1}{s_1},...,\frac{a_n}{s_n}) \in S^{-1}I$ for $\frac{a_1}{s_1},...,\frac{a_n}{s_n} \in S^{-1}R$.
Then we have $\frac{g(a_1^n)}{g(s_1^n)} \in S^{-1}I$. It implies that there exists $t \in S$ such that $g(t,g(a_1^n),1^{(n-2)}) \in I$. Since $I$ is an $n$-ary $S$-prime hyperideal of $R$ and $I \cap S=\varnothing$, then  there exists $s \in S$ such that $g(s,a_i,1^{(n-2)}) \in I$ or $g(s,g(t,a_i,1^{(n-2)}),1^{(n-2)})=g(s,t,a_i,1^{(n-3)}) \in I$   for some $1 \leq i \leq n$. Hence we conclude that $G(\frac{a_i}{s_i},\frac{1}{1}^{(n-1)})=\frac{g(a_i,1^{(n-1)})}{g(s_i,1^{(n-1)}))}=\frac{g(s,a_i,1^{(n-2)})}{g(s,s_i,1^{(n-2)})}  \in S^{-1}I$ or  $G(\frac{a_i}{s_i},\frac{1}{1}^{(n-1)})=\frac{g(a_i,1^{(n-1)})}{g(s_i,1^{(n-1)})}=\frac{g(s,t,a_i,1^{(n-3)})}{g(s,t,s_i,1^{(n-2)})}  \in S^{-1}I$ for some $1 \leq i \leq n$ . Thus $S^{-1}I$ is an $n$-ary prime hyperideal of $S^{-1}R$.
\end{proof}
\begin{theorem}
Let $R$ be a Krasner $(m,n)$-hyperring, $S$ be an $n$-ary multiplicative subset of $R$ with $1 \in S$ and $I$ be a hyperideal of $R$ with $I \cap S=\varnothing$ . If $S^{-1}I$ is an $n$-ary prime hyperideal of $S^{-1}R$ and $S^{-1}I \cap R=(I:s)$ for some $ s \in S$, then $I$ is an $n$-ary $S$-prime hyperideal of $R$.
\end{theorem}
\begin{proof}
Let $S^{-1}I$ be an $n$-ary prime hyperideal of $S^{-1}R$ and $S^{-1}I \cap R=(I:s)$ for some $ s \in S$. Assume that $g(a_1^n) \in I$ for some $a_1^n \in R$. Then we get $G(\frac{a_1}{1},\cdots,\frac{a_n}{1}) \in S^{-1}I$. Since $S^{-1}I$ is an $n$-ary prime hyperideal of $S^{-1}R$, we obtain $\frac{a_i}{1} \in S^{-1}I$ for some $1 \leq i \leq n$ which implies $g(t,a_i,1^{(n-2)}) \in I$ for some $t \in S$. Hence $a_i=\frac{g(t,a_i,1^{(n-2)})}{g(t,1^{(n-1)})} \in S^{-1}I$. This means $a_i \in (I:s)$ for somr $s \in S$. Therefore we have $g(s,a_i,1^{(n-2)}) \in I$. Thus we conclude that $I$ is an $n$-ary $S$-prime hyperideal of $R$.
\end{proof}
Let $J$ be a hyperideal of a Krasner $(m, n)$-hyperring $(R, f, g)$. Then the set

$R/J = \{f(x^{i-1}_1, J, x^m_{i+1}) \ \vert \  x^{i-1}_1,x^m_{i+1} \in R \}$\\
endowed with m-ary hyperoperation $f$ which for all $x_{11}^{1m},...,x_{m1}^{mm} \in R$

$f(f(x_{11}^{1 (i-1)},J, x^{1m}_ {1(i+1)}),..., f(x_{m1}^{ m(i-1)}, J, x^{mm}_ {m(i+1)}))$ 

$= f (f(x^{m1}_{11}),..., f(x^{m(i-1)}_{1(i-1)}), J, f(x^{m(i+1)}_{1(i+1)} ),..., f(x^{mm}_ {1m}))$\\
and with $n$-ary hyperoperation g which for all $x_{11}^{1m},...,x_{n1}^{nm} \in R$

$g(f(x_{11}^{1 (i-1)}, J, x^{1m}_ {1(i+1)}),..., f(x_{n1}^{ n(i-1)}, J, x^{nm}_ {n(i+1)}))$ 

$= f (g(x^{n1}_{11}),..., g(x^{n(i-1)}_{1(i-1)}), J, g(x^{n(i+1)}_{1(i+1)} ),..., f(x^{nm}_ {1m}))$\\
construct a Krasner $(m, n)$-hyperring, and $(R/J, f, g)$  is called the quotient Krasner $(m, n)$-hyperring of $R$ by $J$ \cite{sorc1}. 

\begin{theorem} \label{51}
Let $S$ be  an n-ary multiplicative subset of a Krasner $(m,n)$-hyperring $R$ and let $I$ and $J$ be a hyperideals of $R$ such that $J \subseteq I$. Let $J \cap S=\varnothing$ and $I/J \cap \bar{S}=\varnothing$ with $\bar{S}=\{f(s_1^{i-1},J,s_{i+1}^n) \ \vert \ s_1^{i-1},s_{i+1}^n \in S\}$. If $I$ is an  n-ary $S$-prime hyperideal of $R$, then $I/J$ is an n-ary $\bar{S}$-prime hyperideal of $R/J$.
\end{theorem}
\begin{proof} Let $I$ be an  n-ary $S$-prime hyperideal of $R$. Then there exists some $s \in S$ such that if $g(x_1^n) \in I$ for $x_1^n \in R$, then we get $g(s,x_i,1^{(n-2)}) \in I$ for some $1 \leq i \leq n$.
Let 

$g(f(x_{11}^{1(i-1)},J,x_{1(i+1)}^{1m}),...,f(x_{n1}^{n(i-1)},J,x_{n(i+1)}^{nm})) \in I/J$\\ for some $f(x_{11}^{1(i-1)},J,x_{1(i+1)}^{1m}),...,f(x_{n1}^{n(i-1)},J,x_{n(i+1)}^{nm}) \in R/J$. \\This means 

$f(g(x_{11}^{n1}),...,g(x_{1(i-1)}^{n(i-1)}),J,g(x_{1(i+1)}^{n(i+1)}),...,g(x_{1m}^{nm})) \in I/J$.\\ Then $f(g(x_{11}^{n1}),...,g(x_{1(i-1)}^{n(i-1)}),0,g(x_{1(i+1)}^{n(i+1)}),...,g(x_{1m}^{nm})) \subseteq I$ which implies

 $g(f(x_{11}^{1(i-1)},0,x_{1(i+1)}^{1m}),...,f(x_{n1}^{n(i-1)},0,x_{n(i+1)}^{nm}) ) \subseteq  I$.\\
Since $I$ is an  n-ary $S$-prime hyperideal of $R$, then, for some $1 \leq j \leq n$, we obtain

  $g(s,f(x_{j1}^{j(i-1)},0,x_{j(i+1)}^{jm}),1^{n-2}) \subseteq I$ \\  Therefore 
  
  $f(g(s,f(x_{j1}^{j(i-1)},0,x_{j(i+1)}^{jm}),1^{n-2}),J,0^{(m-2)}) \in I/J$ \\and so

  $f(g(g(s,1^{(n-2)}),f(x_{j1}^{j(i-1)},0,x_{j(i+1)}^{jm}),1^{n-2}),J,0^{(m-2)}) \in I/J$.\\
  Thus we conclude that

  $g(f(s,J,1^{(n-2)}),f(x_{j1}^{j(i-1)},J,x_{j(i+1)}^{jm}),1_{R/J}^{(n-2)}) \in I/J$
\\Consequently, $I/J$ is an n-ary $\bar{S}$-prime hyperideal of $R/J$.

\end{proof}

Suppose that $I$ is a normal hyperideal of Krasner $(m,n)$-hyperring $(R,f,g)$. Then the set of all equivalence classes $[R:I^*]=\{I^*[x] \ \vert \ x \in R\}$ is a Krasner $(m,n)$-hyperring with the m-ary hyperoperation  $f/I$ and the n-ary operation $g/I$, defined as follows:

$f/I(I^*[x_1], \cdots,I^*[x_m])=\{I^*[z] \ \vert \ z \in f(I^*[x_1], \cdots, I^*[x_m])\}, \ \ \forall x_1^m \in R$

$g/I(I^*[x_1], \cdots,I^*[x_n])=I^*[g(x_1^n)], \ \ \forall x_1^n \in R$\\(for more details refer to \cite{d1}). Now, we establish the following result.
\begin{theorem}
Let $S$ be  an n-ary multiplicative subset of a Krasner $(m,n)$-hyperring $(R,f,g)$ and let  $I$ be a normal hyperideal of $R$. If $J$ is an n-ary $S$-prime hyperideal of $R$ such that $I \subseteq J$, then $[J:I^*]$ is an n-ary $[S:I^*]$-prime hyperideal of $[R:I^*]$
\end{theorem}
\begin{proof}
First of all, notice that $S \cap J = \varnothing$ if and only if $[S:I^*] \cap [J:I^*]=\varnothing$. Let $g/I(I^*[x_1],\cdots,I^*[x_n]) \in [J:I^*]$ for some $x_1^n \in R$. Then $I^*[g(x_1^n)] \in [J:I^*]$. This means  $I^*[g(x_1^n)] \subseteq J$. So

$I^*[g(x_1^n)]=f(I,g(x_1^n),0^{(m-2)})=f(I,g(x_1^n),g(0^{(n)})^{(m-2)})$

$\hspace{1.5cm}=g(f(I,x_1,0^{(m-2)}),\cdots,f(I,x_n,0^{(m-2)})) \subseteq J$\\
Since $J$ is an n-ary $S$-prime hyperideal of $R$, then there exists $s \in S$ such that $g(s,f(I,x_i,0^{(m-2)}),1^{(n-2)}) \subseteq J$ for some $ 1 \leq i \leq n$ which implies

 $f(I,g(s,x_i,1^{(n-2)}),0^{(m-2)}) \subseteq J$. \\Hence $I^*[g(s,x_i,1^{(n-2)})] \in [J:I^*]$ which means $g/I(I^*[s],I^*[x_i],I^*[1]^{(n-2)}) \in [J:I^*]$. Thus $[J:I^*]$ is an n-ary $[S:I^*]$-prime hyperideal of $[R:I^*]$.
\end{proof}
Let $(R_1, f_1, g_1)$ and $(R_2, f_2, g_2)$ be two Krasner $(m,n)$-hyperrings such that $1_{R_1}$ and $1_{R_2}$ be scalar identitis of $R_1$ and $R_2$, respectively. Then 
the $(m, n)$-hyperring $(R_1 \times R_2, f_1\times f_2 ,g_1 \times g_2 )$ is defined by m-ary hyperoperation
$f=f_1\times f_2 $ and n-ary operation $g=g_1 \times g_2$, as follows:

$f_1 \times f_2((a_{1}, b_{1}),\cdots,(a_m,b_m)) = \{(a,b) \ \vert \ \ a \in f_1(a_1^m), b \in f_2(b_1^m) \}$

$g_1 \times g_2 ((x_1,y_1),\cdots,(x_n,y_n)) =(g_1(x_1^n),g_2(y_1^n)) $,\\
for all $a_1^m,x_1^n \in R_1$ and $b_1^m,y_1^n \in R_2$ \cite{mah2}. Suppose that  $S=S_1 \times S_2$ such that $S_1$ and $S_2$ are n-ary multiplicative subsets of  $R_1$ and $R_2$, respectively. Assume that $I_1$ is an n-ary $S_1$-prime hyperideal of $R_1$. It is easy to verify that $(I_1 \times R_2) \cap S=\varnothing  \Longleftrightarrow   I_1 \times S_1 = \varnothing$. In the next  theorem, we characterize n-ary $S$-prime hyperideals of $R_1 \times R_2$.
\begin{theorem}
Let $(R_1, f_1, g_1)$ and $(R_2, f_2, g_2)$ be two Krasner $(m,n)$-hyperrings such that $1_{R_1}$ and $1_{R_2}$ be scalar identitis of $R_1$ and $R_2$, respectively. Suppose that  $S=S_1 \times S_2$ such that $S_1$ and $S_2$ are n-ary multiplicative subsets of   $R_1$ and $R_2$, respectively. Then  $I_1$ is an n-ary $S_1$-prime hyperideal of $R_1$ if and only if  $I_1 \times R_2$ is an n-ary $S$-prime hyperideal of $R_1 \times R_2$.
\end{theorem}
\begin{proof}
$\Longrightarrow$ Assume that  $I_1$ is  an n-ary $S_1$-prime hyperideal  of $R_1$. Let $g_1 \times g_2((x_1,y_1), \cdots, (x_n,y_n)) \in I_1 \times R_2$ for some $x_1^n \in R_1$ and $y_1^n \in R_2$. Then we get $g_1(x_1^n) \in I_1$. By the hypothesis, there exists $s_1 \in S_1$ such that $g_1(s_1,x_i,1^{(n-2)}) \in I_1$ for some $1 \leq i \leq n$. Then we have $g_1 \times g_2((s_1,1_{R_2}),(x_i,y_i),(1_{R_1},1_{R_2})^{(n-2)}) \in I_1 \times R_2$. Thus we conclude that $I_1 \times R_2$ is an n-ary $S$-prime hyperideal of $R_1 \times R_2$.\\
$\Longleftrightarrow$ Let $I_1 \times R_2$ be an n-ary $S$-prime hyperideal of $R_1 \times R_2$. Assume that $g_1(x_1^n) \in I_1$ for some $x_1^n \in R_1$. Then $g_1 \times g_2((x_1,1_{R_2}),\cdots (x_n,1_{R_2})) \in I_1 \times R_2$. Since  $I_1 \times R_2$ is an n-ary $S$-prime hyperideal of $R_1 \times R_2$, then there exists an element $(s_1, s_2)$ in $S$ such that  $g_1 \times g_2 ((s_1,s_2),(x_i,1_{R_2}),(1_{R_1},1_{R_2})^{(n-2)}) \in I_1 \times R_2$ for some $1 \leq i \leq n$. This means $g_1(s_1,x_i,1_{R_1}^{(n-2)}) \in I_1$. Consequently, $I_1$ is an n-ary $S_1$-prime hyperideal of $R_1$.
\end{proof}
We give the following results obtained by the previous theorem.
\begin{corollary}
Let $(R_i, f_i, g_i)$  be a Krasner $(m,n)$-hyperring for each $1 \leq i \leq t$ such that $1_{R_i}$ is scalar identity of $R_i$. Assume that  $S=S_1 \times \cdots \times S_t$ such that $S_i$  is an n-ary multiplicative subset of   $R_i$ for each $1 \leq i \leq t$. If $I_i$ is an n-ary $S_i$-prime hyperideal of $R_i$ for some $1 \leq i \leq t$, then   $R_1 \times \cdots \times R_{i-1} \times I_i \times R_{i+1} \times \cdots \times R_t $ is an n-ary $S$-prime hyperideal of $R_1 \times \cdots \times R_t$.
\end{corollary}
%44444444444444444444444444444444444444444444444444444
%44444444444444444444444444444444444444444444444444444444
\section{$n$-ary S-primary hyperideals}
The aim of this section is to define the notion of n-ary $S$-primary hyperideals in a Krasner $(m,n)$-hyperring. The overall
framework of the structure is then explained.
\begin{definition}
Let $S$ be  an n-ary multiplicative subset of a Krasner $(m,n)$-hyperring $R$ and $I$ be a hyperideal of $R$ with $I \cap S=\varnothing$.  $I$ refers to an n-ary $S$-primary hyperideal  if there exists an $s \in S$ such that for all $x_1^n \in R$ with $g(x_1^n) \in I$, we have $g(s,x_i,1^{(n-2)}) \in I$ or $g(x_1^{i-1},s,x_{i+1}^n) \in \sqrt{I}^{(m,n)}$.
\end{definition}
\begin{example}
In Example \ref{ex}, the hyperideal $P=\{0\}$ is a 3-ary $S$-primary hyperideal of $R$.
\end{example}
 The following is a direct consequence and can be proved easily and
so the proof is omited.
\begin{theorem}
Let $S$ be  an n-ary multiplicative subset of a Krasner $(m,n)$-hyperring $R$ such that $S \subseteq U(R)$ and $I$ be a hyperideal of $R$. Then $I$ is an n-ary $S$-primary hyperideal of $R$ if and only if $I$ is an n-ary primary hyperideal of $R$.
\end{theorem}
\begin{theorem}
Let $S$ be  an n-ary multiplicative subset of a Krasner $(m,n)$-hyperring $R$ and $I$ be a hyperideal of $R$ with $I \cap S=\varnothing$. Then $I$ is an n-ary $S$-primary hyperideal of $R$ if and only if $(I:s)=\{x \in R \ \vert \ g(x,s,1^{(n-2)}) \in I\}$ is an n-ary primary hyperideal of $R$ for some $s \in S$.
\end{theorem}
\begin{proof}
$\Longrightarrow $ Let $I$ be an n-ary $S$-primary hyperideal of $R$. Then there exists $s \in S$ such that if $g(a_1^n) \in I$, then $g(s,a_i,1^{(n-2)}) \in I$ or $g(a_1^{i-1},s,a_{i+1}^n) \in  \sqrt{I}^{(m,n)}$. Now we show $(I:s)=\{x \in R \ \vert \ g(s,x,1^{(n-2)}) \in I\}$ is an n-ary primary hyperideal of $R$. Suppose that $g(x_1^n) \in (I:s)$ for some $x_1^n \in R$. This means $g(g(x_1^n),s,1^{(n-2)})=g(x_1^{i-1},g(s,x_i,1^{(n-2)}),x_{i+1}^n) \in I$. Then we have $g(s,g(s,x_i,1^{(n-2)}),1^{(n-2)})=g(s^{(2)}, x_i,1^{(n-3)}) \in I$ or $g(x_1^{i-1},s,x_{i+1}^n) \in \sqrt{I}^{(m,n)}$. In the first case, we have $g(s,x_i,1^{(n-2)}) \in I$ or $g(s^{(3)},1^{(n-3)}) \in \sqrt{I}^{(m,n)}$. From $g(s^{(3)},1^{(n-3)}) \in \sqrt{I}^{(m,n)}$, it follows that there exists $t \in \mathbb{N}$ such that  $g(g(s^{(3)},1^{(n-3)})^{(t)},1^{(n-t)}) \in I$ for $t \leq n$ or $g_{(l)}(g(s^{(3)},1^{(n-3)})^{(t)}) \in I$ for $t=l(n-1)+1$. In both possibilities, we conclude that $I \cap S \neq \varnothing$ which is a contradiction. Hence get $g(s,x_i,1^{(n-2)}) \in I$ which implies $x_i \in (I:s)$. In the second case, there exists $t \in \mathbb{N}$ such that $g(g(x_1^{i-1},s,x_{i+1}^n)^{(t)},1^{(n-t)}) \in I$ for $t \leq n$ or $g_{(l)}(g(x_1^{i-1},s,x_{i+1}^n)^{(t)}) \in I$ for $t=l(n-1)+1$. If $t \leq n$, then we obtain $g(s,g(s^{(t)},1^{(n-t)}),1^{(n-2)})=g(s^{(t+1)},1^{(n-t-1)}) \in \sqrt{I}^{(m,n)}$ or $g(s,g(x_1^{i-1},1,x_{i+1}^n)^{(t)},1^{(n-t-1)}) \in I$. From $g(s^{(t+1)},1^{(n-t-1)}) \in \sqrt{I}^{(m,n)}$, it follows that $I \cap S \neq \varnothing$. Thus we conclude that $g(s,g(x_1^{i-1},1,x_{i+1}^n)^{(t)},1^{(n-t-1)}) \in I$ which means $g(g(x_1^{i-1},1,x_{i+1}^n)^{(t)},1^{(n-t)}) \in (I:s)$. Therefore $g(x_1^{i-1},1,x_{i+1}^n) \in \sqrt{(I:s)}$. Similar for other case. Consequently, $(I:s)$ is an n-ary primary hyperideal of $R$.

$\Longleftarrow$ Let $(I:s)=\{x \in R \ \vert \ g(x,s,1^{(n-2)}) \in I\}$ be an n-ary primary hyperideal of $R$ for some $s \in S$. Suppose that $g(x_1^n) \in I$ for some $x_1^n \in R$. Then we get $g(x_1^n) \in (I:s)$. Since $(I:s)$ is an n-ary primary hyperideals of $R$, then $x_i \in E_S$ which implies $g(s,x_i,1^{(n-2)}) \in I$ or $g(x_1^{i-1},1,x_{i+1}^n) \in \sqrt{(I:s)}$ which means there exists $t \in \mathbb{N}$ such that $g(s,g(x_1^{i-1},1,x_{i+1}^n)^{(t)},1^{(n-t-1)}) \in I$ for $t \leq n$ or $g(s,g_{(l)}(x_1^{i-1},1,x_{i+1}^n)^{(t)}),1^{(n-2)}) \in I$ for $t=l(n-1)+1$. If $t \leq n$, then we have $g(g(x_1^{i+1},s,x_{i+1}^n)^{(t)},1^{(n-t)}) \in I$ which implies $g(x_1^{i-1},s,x_{i+1}^n) \in \sqrt{I}$. Similar for other case. This means $I$ is an n-ary $S$-primary hyperideal of $R$.
\end{proof}
In the following, we consider the relationship between an n-ary $S$-primary hyperideal and its radical.
\begin{theorem}
Let $S$ be  an n-ary multiplicative subset of a Krasner $(m,n)$-hyperring $R$ and $I$ be an n-ary $S$-primary hyperideal of $R$. Then $\sqrt{I}^{(m,n)}$ is an n-ary $S$-prime hyperideal of $R$.
\end{theorem}
\begin{proof}
Since $I \cap S = \varnothing$ then we conclude that $\sqrt{I}^{(m,n)} \cap S=\varnothing$. Now, let $g(x_1^n) \in \sqrt{I}$ and for all $j \in \{1,...,i-1,i+1,...,n\}$, $g(s,x_j,1^{(n-2)}) \notin \sqrt{I}^{(m,n)}$ for each $s \in S$.  Since $g(x_1^n) \in \sqrt{I}^{(m,n)}$, there exists $t \in \mathbb{N}$ such that $g(g(x_1^n)^{(t)},1^{(n-t)}) \in I$ for $t \leq n$ or $g_{(l)}(g(x_1^n)^{(t)}) \in I$ for $t=l(n-1)+1$. If $g(g(x_1^n)^{(t)},1^{(n-t)}) \in I$, then there exists $s \in S$ such that $g(s,x_i^{(t)},1^{(n-t-1)}) \in I$ or $g(s,g(x_1^{i-1},1,x_{i+1}^n)^{(t)},1^{(n-t-1)})$ as $I$ is a n-ary $S$-primary hyperideal of $R$. Hence we get $g(g(s,x_i,1^{(n-2)})^{(t)},1^{(n-t)}) \in I$ which implies $g(s,x_i,1^{(n-2)}) \in \sqrt{I}^{(m,n)}$ or $g(g(x_1^{i-1},s,x_{i+1}^n)^{(t)},1^{(n-t)}) \in I$. In the second case, we get 

$\hspace{0.7cm}g(g(x_1^{(t)},g(x_2^{i-1},s,1,x_{i+1}^n)^{(t)},1^{(n-2t)}) \in I$

$\Longrightarrow g(g(x_1^{(t)},1^{(n-t)}),g(x_2^{i-1},s,1,x_{i+1}^n)^{(t)},1^{(n-t-1)}) \in I$

$\Longrightarrow g(s,g(x_1^{(t)},1^{(n-t)}),1^{(n-2)}) \in I \ \text{or}\  g(s,g(x_2^{i-1},s,1,x_{i+1}^n)^{(t)},1^{(n-t-1)}) \in \sqrt{I}^{(m,n)}$

$\Longrightarrow g(g(s,x_1,1^{(n-2)})^{(t)},1^{(n-t)}) \in I \ \text{or}\  g(g(x_2^{i-1},s^{(2)},x_{i+1}^n)^{(t)},1^{(n-t)}) \in \sqrt{I}^{(m,n)}$

$\Longrightarrow g(s,x_1,1^{(n-2)}) \in \sqrt{I}^{(m,n)} \ \text{or}\  g(g(x_2^{i-1},s^{(2)},x_{i+1}^n)^{(t)},1^{(n-t)}) \in \sqrt{I}^{(m,n)}$\\
Since $g(s,x_1,1^{(n-2)}) \in \sqrt{I}^{(m,n)}$ is a contradiction, then 

$ \hspace{0.7cm}g(g(x_2^{i-1},s^{(2)},x_{i+1}^n)^{(t)},1^{(n-t)}) \in \sqrt{I}^{(m,n)}$

$\Longrightarrow \exists w \in \mathbb{N}; g(g(g(x_2^{i-1},s^{(2)},x_{i+1}^n),1^{(n-t)})^{(w)}, 1^{(n-w)}) \in I$

$\Longrightarrow g(g(x_2^{(t+w)},1^{(n-t-w)},g(g(x_3^{i-1},s^{(2)},1,x_{i+1}^n)^{(w)},1^{(n-w-1)}) \in I$

$\Longrightarrow g(s,x_2,1^{(n-2)}) \in \sqrt{I}^{(m,n)} \ \text{or} \ g(g(g(x_3^{i-1},s^{(3)},x_{i+1}^n)^{(t)},1^{(n-t)})^{(w)},1^{(n-w)}) \in $

$\hspace{0.7cm}\sqrt{I}^{(m,n)}$

$\vdots$

$\Longrightarrow \cdots \ \text{or} \  g(s,x_n,1^{(n-2)}) \in \sqrt{I}^{(m,n)}$\\
which is contradiction with $g(s,x_j,1^{(n-2)}) \notin \sqrt{I}^{(m,n)}$ for all $j \in \{1,...,i-1,i+1,...,n\}$. Thus we have 
$g(s,x_i,1^{(n-2)}) \in \sqrt{I}^{(m,n)}$.  Consequently, $\sqrt{I}^{(m,n)}$ is an n-ary $S$-prime hyperideal of $R$. If $t=l(n-1)+1$, then by using a similar argument, one can easily complete the proof.
\end{proof}
For a hyperideal $I$ of a Krasner $(m,n)$-hyperring $R$, we refer to the n-ary $S$-prime hyperideal $P=\sqrt{I}^{(m,n)}$ as the associated $S$-prime hyperideal of $I$ and on the other hand $I$ is referred to as an n-ary $P$-$S$-primary hyperideal of $R$.
\begin{theorem}
Let $S$ be  an n-ary multiplicative subset of a Krasner $(m,n)$-hyperring $R$ and $I_1^n$ be n-ary $P$-$S$-primary hyperideals of $R$ for some n-ary $S$-prime hyperideal $P$ of $R$. Then, $I=\cap_{i=1}^nI_i$ is an n-ary $P$-$S$-primary hyperideal of $R$.
\end{theorem}
\begin{proof}
Let $I_j$ be an n-ary $S$-primary hyperideal of $R$ for all $1 \leq j \leq n$. Then there exists $s_j \in S$ such that if $g(x_1^n) \in I_j$ for $x_1^n \in R$, then $g(s_j,x_i,1^{(n-2)}) \in I_j$ or $g(x_1^{i-1},s_j,x_{i+1}^n) \in  \sqrt{I_j}^{(m,n)}$ for some $1 \leq i \leq n$. We put $s=g(s_1^n)$. Let $g(a_1^n) \in I$ for some $a_1^n \in R$ and let us assume that $g(s,a_i,1^{(n-2)}) \notin I$ for all $1 \leq i \leq n$. This means $g(s_j,a_i,1^{(n-2)}) \notin I_j$ for some $1 \leq j \leq n$. Since $I_j$ is an $S$-primary hyperideal of $R$, we get $g(a_1^{i-1},s_j,a_{i+1}^n) \in \sqrt{I_j}^{(m,n)}$. Since $I_1^n$ are n-ary $P$-$S$-primary hyperideals of $R$ for some n-ary $S$-prime hyperideal $P$ of $R$, then $\sqrt{I_j}^{(m,n)}=\sqrt{I}$. Then  we obtain $g(a_1^{i-1},s_j,a_{i+1}^n) \in \sqrt{I}^{(m,n)}$ which implies $g(a_1^{i-1},s,a_{i+1}^n) \in \sqrt{I}^{(m,n)}$ and the proof is over.
\end{proof}
\begin{theorem}
Let $S$ be  an n-ary multiplicative subset of a Krasner $(m,n)$-hyperring $R$ and $I_1^{n-1}$ be hyperideals of $R$ such that for each $1 \leq i \leq n-1$, $I_i \cap S \neq \varnothing$. If $I$ is an n-ary $S$-primary hyperideal of $R$, then $g(I_1^{n-1},I)$ is an n-ary $S$-primary hyperideal of $R$.
\end{theorem}
\begin{proof}
Since $g(I_1^{n-1},I) \subseteq I$ and $I \cap S =\varnothing$, we get $g(I_1^{n-1},I) \cap S=\varnothing$. Let $I$ be an n-ary $S$-primary hyperideal of $R$. Then there exists $s \in S$ such that if $g(a_1^n) \in I$ for $a_1^n \in R$, then $g(s,a_i,1^{(n-2)}) \in I$ or $g(a_1^{i-1},s,a_{i+1}^n) \in  \sqrt{I}^{(m,n)}$. Suppose that $g(x_1^n) \in g(I_1^{n-1},I)$ for some $x_1^n \in R$. As $g(I_1^{n-1},I) \subseteq I$, then $g(x_1^n) \in I$. Since for each $1 \leq i \leq n-1$, $I_i \cap S \neq \varnothing$, then we have $u_i \in I_i \cap S$. Put $u=g(u_1^{n-1},1)$. If $g(s,x_i,1^{(n-2)}) \in I$, then $g(g(u,s,1^{(n-2)}),x_i,1^{(n-2)})=g(u,g(s,x_i,1^{(n-2)}),1^{(n-2)}) \in g(I_1^{(n-1)},I)$. Now, let $g(x_1^{i-1},s,x_{i+1}^n) \in \sqrt{I}^{(m,n)}$. Then there exists $t \in \mathbb{N}$ such that $g(g(x_1^{i-1},s,x_{i+1}^n)^{(t)},1^{(n-t)}) \in I$ for $t \leq n$ or $g_{(l)}(g(x_1^{i-1},s,x_{i+1}^n)^{(t)}) \in I$ for $t=l(n-1)+1$. In the former case, we have

 $\hspace{0.7cm} g(u^{(t)},g(g(x_1^{i-1},s,x_{i+1}^n)^{(t)},1^{(n-t)})),1^{(n-t-1)}) $ 

$\hspace{0.4cm}=g(g(x_1^{i-1},g(u,s,1^{(n-2)}),x_{i+1}^n)^{(t)},1^{(n-t)}) $

 $\hspace{0.4cm}\in g(I_1^{n-1},I)$\\
which implies $g(x_1^{i-1},g(u,s,1^{(n-2)}),x_{i+1}^n) \in \sqrt{g(I_1^{n-1},I)}^{(m,n)}$. Thus, $g(I_1^{n-1},I)$ is an n-ary $S$-primary hyperideal of $R$. In the second case, a similar argument completes the proof. 
\end{proof}
\begin{theorem} 
Let $R$ be a Krasner $(m,n)$-hyperring and $S$ be an $n$-ary multiplicative subset of $R$ with $1 \in S$. If $I$ is an $n$-ary $S$-primary hyperideal of $R$ with $I \cap S=\varnothing$, then $S^{-1}I$ is an $n$-ary primary hyperideal of $S^{-1}R$.
\end{theorem}
\begin{proof}
Let $\frac{a_1}{s_1},...,\frac{a_n}{s_n} \in S^{-1}R$ such that $G(\frac{a_1}{s_1},...,\frac{a_n}{s_n}) \in S^{-1}I$ .
Then we have $\frac{g(a_1^n)}{g(s_1^n)} \in S^{-1}I$. It implies that there exists $t \in S$ such that $g(t,g(a_1^n),1^{(n-2)}) \in I$ and so $g(a_1^{i-1},g(t,a_i,1^{(n-2)}),a_{i+1}^n) \in I$. Without destroying the generality, we may assume that $g(a_1^{n-1},g(t,a_n,1^{(n-2)})) \in I$. Since $I$ is an $n$-ary $S$-primary hyperideal of $R$, then there exist $s \in S$ such that at least one of the cases hold: $g(s,a_i,1^{(n-2)}) \in I$ for some $1 \leq i \leq n-1$,  $g(s,g(t,a_n,1^{(n-2)}),1^{(n-2)}) \in I$,  $g(a_1^{i-1},s,a_{i+1}^{n-1},g(t,a_n,1^{(n-2)})) \in \sqrt{I}^{(m,n)}$ for some $1 \leq i \leq n-1$ or $g(a_1^{n-1},s) \in \sqrt{I}^{(m,n)}$.\\ 
If $g(s,a_i,1^{(n-2)}) \in I$ for some $1 \leq i \leq n-1$, then $G(\frac{a_i}{s_i},\frac{1}{1}^{(n-1)})=
\frac{g(a_i,1^{(n-1)})}{g(s_i,1^{(n-1)})}=\frac{g(s,a_i,1^{(n-2)})}{g(s,s_i,1^{(n-2)})} \in S^{-1}I$. 
If $g(s,g(t,a_n,1^{(n-2)}),1^{(n-2)}) \in I$ then $G(\frac{a_n}{s_n},\frac{1}{1}^{(n-1)})=
\frac{g(a_n,1^{(n-1)})}{g(s_n,1^{(n-1)})}=\frac{g(s,t, a_n,1^{(n-3)})}{g(s,t,s_n,1^{(n-3)})} \in S^{-1}I$. If $g(a_1^{i-1},s,a_{i+1}^{n-1}, g(t,a_n,1^{(n-2)}) \in \sqrt{I}^{(m,n)}$ for some $1 \leq i \leq n-1$, then we get  $G(\frac{a_1}{s_1},...,\frac{a_{i-1}}{s_{i-1}},\frac{1}{1},\frac{a_{i+1}}{s_{i+1}},...,\frac{a_{n-1}}{s_{n-1}},\frac{a_n}{s_n})= \break \frac{g(a_1^{i-1},g(s,t,1^{(n-2)}),a_{i+1}^{n-1},a_n)}{g(s_1^{i-1},g(s,t,1^{(n-2)}),s_{i+1}^{n-1},a_n)} \in S^{-1}\sqrt{I}^{(m,n)}=\sqrt{S^{-1}I}^{(m,n)}$, by Lemma 4.7 in  \cite{mah5}. If $g(a_1^{n-1},s) \in \sqrt{I}^{(m,n)}$, then $G(\frac{a_1}{s_1},...,\frac{a_{n-1}}{s_{n-1}},\frac{1}{1})=\frac{g(a_1^{n-1},s)}{g(s_1^{n-1},s)} \in S^{-1}\sqrt{I}^{(m,n)}=\sqrt{S^{-1}I}^{(m,n)}$. Thus $S^{-1}I$ is an $n$-ary primary hyperideal of $S^{-1}R$.
\end{proof}
\begin{theorem}
Let $R$ be a Krasner $(m,n)$-hyperring, $S$ be an $n$-ary multiplicative subset of $R$ such that $1 \in S$ and $I$ be a hyperideal of $R$ with $I \cap S=\varnothing$ . If $S^{-1}I$ is an $n$-ary primary hyperideal of $S^{-1}R$ and $S^{-1}I \cap R=(I:s)$ for some $ s \in S$, then $I$ is an $n$-ary $S$-primary hyperideal of $R$.
\end{theorem}
\begin{proof}
Suppose that  $S^{-1}I$ is an $n$-ary primary hyperideal of $S^{-1}R$ and $S^{-1}I \cap R=(I:s)$ for some $ s \in S$. Assume that $g(a_1^n) \in I$ for some $a_1^n \in R$. Then we obtain $G(\frac{a_1}{1},\cdots,\frac{a_n}{1}) \in S^{-1}I$. Since $S^{-1}I$ is an $n$-ary primary hyperideal of $S^{-1}R$, we get $\frac{a_i}{1} \in S^{-1}I$ or $G(\frac{a_1}{1},\cdots,\frac{a_{i-1}}{1},\frac{1}{1},\frac{a_{i+1}}{1},\cdots,\frac{a_n}{1}) \in \sqrt{S^{-1}I}^{(m,n)}$ for some $1 \leq i \leq n$. In the former case, we have $a_i=\frac{a_i}{1} \in S^{-1}I \cap R$ which implies $a_i \in (I:s)$ by the hypothesis. This means $g(s,a_i,1^{(n-2)})\in I$ and we are done. In the second case, we get $G(G(\frac{a_1}{1},\cdots,\frac{a_{i-1}}{1},\frac{1}{1},\frac{a_{i+1}}{1},\cdots,\frac{a_n}{1})^{(t)},\frac{1}{1}^{(n-t)}) \in S^{-1}I$ for $t \leq n$ or $G_{(l)}(G(\frac{a_1}{1},\cdots,\frac{a_{i-1}}{1},\frac{1}{1},\frac{a_{i+1}}{1},\cdots,\frac{a_n}{1})^{(t)}) \in S^{-1}I$ for $t=l(n-1)+1$. The first possibility follows that $g(g(a_1,\cdots,a_{i-1},1,a_{i+1},\cdots,a_n)^{(t)},1^{(n-t)})=\frac{g(g(a_1,\cdots,a_{i-1},1,a_{i+1},\cdots,a_n)^{(t)},1^{(n-t)})}{g(g(1^{(n)})^{(t)},1^{(n-t)})} \in S^{-1}I \cap R$. Hence we conclude that $\break$ $g(s, g(a_1,\cdots,a_{i-1},1,a_{i+1},\cdots,a_n)^{(t)},1^{(n-t-1)}) \in I$, By the hypothesis. This means $g( g(a_1,\cdots,a_{i-1},s,a_{i+1},\cdots,a_n)^{(t)},1^{(n-t)}) \in I$ which implies $\break$ $g(a_1,\cdots,a_{i-1},s,a_{i+1},\cdots,a_n) \in \sqrt{I}^{(m,n)}$, as needed. In the second possibility, 
one can easily complete the proof by using an argument similar.
\end{proof}

%%%%%%%%%%%%%%%%%%%%%%%%%%%%%%%%%%%%%%%%%%%
%%%%%%%%%%%%%%%%%%%%%%%%%

\end{document}